\documentclass[12pt,twoside]{amsart}
\usepackage{mathrsfs}
\usepackage{txfonts}
\usepackage{amsthm}
\usepackage{latexsym}
\usepackage{mathabx}
\usepackage[all]{xy}\SelectTips{eu}{}
\usepackage{amssymb}
\usepackage{amscd,graphicx}
\usepackage{hyperref}

\date{}
\allowdisplaybreaks[4] \footskip=15pt

\renewcommand{\uppercasenonmath}[1]{}

\numberwithin{equation}{section} \theoremstyle{plain}
\newtheorem*{thm*}{Main Theorem}
\newtheorem{thm}{Theorem}[section]
\newtheorem{cor}[thm]{Corollary}
\newtheorem*{cor*}{Corollary}
\newtheorem{lem}[thm]{Lemma}
\newtheorem*{lem*}{Lemma}
\newtheorem{fact}[thm]{Fact}
\newtheorem*{fact*}{Fact}

\newtheorem*{nota*}{Notation}
\newtheorem{prop}[thm]{Proposition}
\newtheorem*{prop*}{Proposition}
\newtheorem{rem}[thm]{Remark}
\newtheorem*{rem*}{Remark}
\newtheorem{exa}[thm]{Example}
\newtheorem*{exa*}{Example}
\newtheorem{df}[thm]{Definition}
\newtheorem*{df*}{Definition}
\newtheorem{setup}[thm]{Setup}

\newtheorem*{ack*}{ACKNOWLEDGEMENTS}

\newcommand{\K}{\mbox{\rm K}}
\newcommand{\Ch}{\mbox{\rm Ch}}

\newcommand{\Hom}{\mbox{\rm Hom}}
\newcommand{\Ext}{\mbox{\rm Ext}}
\newcommand{\dg}{\mbox{\rm dg}}

\newcommand{\id}{\mbox{\rm id}}
\newcommand{\pd}{\mbox{\rm pd}}
\newcommand{\Mod}{\mbox{\rm Mod}}

\newcommand{\Ho}{\mbox{\rm Ho}}

\newcommand{\Rep}{\mbox{\rm Rep}}

\begin{document}
\begin{center}
{\Large \bf Model Structures Arising from Extendable Cotorsion Pairs}
\footnotetext{* Corresponding author.

E-mail: 230238462@seu.edu.cn, ~wangjunpeng1218@163.com, and~z990303@seu.edu.cn}

\vspace{0.5cm} {\small Qingyu Shao$^1$, Junpeng Wang$^2$, Xiaoxiang Zhang$^{1,*}$ \\\

1.  \emph{School of Mathematics}, \emph{Southeast University}, \\ \emph{Nanjing 210096, People's Republic of China} \\

2. \emph{Department of Mathematics}, \emph{Northwest Normal University}, \\ \emph{Lanzhou 730070, People's Republic of China}
}
\end{center}

\vskip.5cm

\noindent{\bf Abstract:}
{The aim of this paper is to construct exact model structures
from so called \emph{extendable} cotorsion pairs.
Given a hereditary Hovey triple $(\mathcal{C}, \mathcal{W}, \mathcal{F})$
in a weakly idempotent complete exact category with enough projectives and injectives.
If one of the cotorsion pairs
$(\mathcal{C}\cap\mathcal{W}, \mathcal{F})$ and
$(\mathcal{C}, \mathcal{W}\cap\mathcal{F})$ is extendable,
then there is a chain of hereditary Hovey triples whose corresponding
homotopy categories coincide.
As applications,
we obtain a new description of the $Q$-shaped derived categories introduced by
Holm and J\o rgensen.
We can also interpret the Krause's recollement
in terms of ``$n$-dimensional'' homotopy categories.
Finally, we have two approaches to get ``$n$-dimensional''
hereditary Hovey triples, which are proved to coincide,
in the category Rep$(Q,\mathcal{A})$ of all representations of
a rooted quiver $Q$ with values in an abelian category $\mathcal{A}$.}

\vskip.2cm

 \vskip.2cm

 \vspace{10pt} \vspace{10pt}

\section{\bf Introduction}

The aim of this paper is to construct exact model structures from so called \emph{extendable} cotorsion pairs (see Definition \ref{D-1}).
It is well known that model structure, introduced by Quillen in \cite{Quillen1967},
is an important tool for formally introducing homotopy theory into a category.
In \cite{Hov2002}, Hovey made a general study of abelian model structures and established a correspondence between abelian model structures and certain (pairs of) cotorsion pairs
in an abelian category.
Such correspondence is referred to as Hovey correspondence in the literature.
Later, Gillespie investigated the exact model structures \cite{Gil11} and extended the Hovey correspondence to the settings of exact categories.

El Maaouy \cite{EL} and Gao \emph{et al} \cite{GZ2024} constructed some chains of model structures from modules of finite Gorenstein dimension.
For example,
El Maaouy proved, among other things, that there exists a hereditary Hovey triple
$(\mathcal{GF}_n, \mathcal{PGF}^{\perp}, \mathcal{C}_n)$
in the category of left modules over a ring $R$ for each non-negative integer $n$
and there are triangulated equivalences
$$\underline{\mathcal{GF}_n\cap \mathcal{C}_n} \simeq \cdots \simeq
\underline{\mathcal{GF}_1\cap \mathcal{C}_1} \simeq
\underline{\mathcal{GF}\cap \mathcal{C}}.$$
Similarly,
Gao \emph{et al} obtained a chain of hereditary Hovey triples
$(\mathcal{PGF}_n, \mathcal{PGF}^{\perp}, \mathcal{P}_n^{\perp})$
whose corresponding homotopy categories coincide.
For more details and related results, we refer the reader to \cite{EL, GZ2024} and the references therein.
Inspired by \cite{EL, GZ2024, Gil11}, we consider model structures arising from extendable cotorsion pairs in a weakly idempotent complete exact category with enough projectives and injectives.
If there is a hereditary Hovey triple $(\mathcal{C}, \mathcal{W}, \mathcal{F})$ and one of the cotorsion pairs,
$(\mathcal{C}\cap\mathcal{W}, \mathcal{F})$ and $(\mathcal{C}, \mathcal{W}\cap\mathcal{F})$, is extendable,
then we can obtain a chain of hereditary Hovey triples whose corresponding homotopy categories coincide (see Theorems \ref{modle structure} and \ref{modle structure 1}).

As applications,
we obtain a new description of the $Q$-shaped derived categories
introduced by
Holm and J\o rgensen \cite{HJ2022} as a generalization of the ordinary derived category $\mathbf{D}(R)$ of a ring $R$ (see Theorem \ref{DQ} and Example \ref{D}).
We can also interpret the Krause's recollement
(see, e.g., \cite{Bec2014, Gil2016, Kra2005})
\[
\xymatrixrowsep{5pc}
\xymatrixcolsep{6pc}
\xymatrix
{\K(\mathrm{ex}\widetilde{\mathcal{I}})  \ar[r]^{} &\K(\mathrm{dw}(\mathcal{I})) \ar@<-3ex>[l]_{} \ar@<3ex>[l]^{} \ar[r]^{}  &\mathbf{D}(R)\ar@<-3ex>[l]_{} \ar@<3ex>[l]^{},
} \]
in terms of ``$n$-dimensional'' homotopy categories
(see Proposition \ref{KR-i}).
Finally, based on \cite[Theorem B]{D-O2021}
and our Theorem \ref{modle structure},
we have two approaches to get ``$n$-dimensional''
hereditary Hovey triples, which are proved to coincide,
in the category Rep$(Q,\mathcal{A})$ of all representations of
a rooted quiver $Q$ with values in an abelian category $\mathcal{A}$
(see Theorems \ref{Q-p} and \ref{Q-i}).

\section{\bf Preliminaries}

Throughout this article,
all rings $R$ are assumed to be associative rings with identity
and all modules are unitary left $R$-modules.
In this section we mainly recall some necessary concepts and facts.
Let $\mathcal{A}$ be an additive category.
Recall that a monomorphism $f: A \to B$ in $\mathcal{A}$ is \emph{split} if there is a morphism $g: B \to A$ such that $gf = 1_A$.
A \emph{split epimorphism} is defined dually.
If every split monomorphism has a cokernel, or equivalently,
every split epimorphism has a kernel,
then $\mathcal{A}$ is said to be \emph{weakly idempotent complete} or \emph{WIC} for short (see \cite{B2010}).

The concept of the exact category is originally due to Quillen \cite{Quillen1973}.
Recall that an \emph{exact category} is an additive category $\mathcal{A}$
equiped with an exact structure $\mathcal{E}$
consisting of a class of \emph{conflations}
$A \rightarrowtail B \twoheadrightarrow C$,
which is a kernel-cokernel pair in $\mathcal{A}$.
Here $A \rightarrowtail B$ (resp., $B \twoheadrightarrow C$)
is called an \emph{inflation} (resp., \emph{deflation}).
One can define the concept of an \emph{acyclic sequence} in terms of
conflations (see \cite[Definition 8.8]{B2010}).
We denote the classes of projectives and injectives by $\mathcal{P}$
and $\mathcal{I}$, respectively.
We recommend \cite{B2010} to interested readers for more details.

\subsection{\bf Cotorsion pairs }

Let $\mathcal{A}$ be an exact category and
$\mathcal{X}$, $\mathcal{Y}$ classes of objects of $\mathcal{A}$.
The pair $(\mathcal{X}, \mathcal{Y})$ is called a
\emph{cotorsion pair} \cite{Gil11}
if $\mathcal{X}^{\perp} = \mathcal{Y}$ and ${^{\perp}\mathcal{Y}} = \mathcal{X}$,
where $\mathcal{X}^{\perp} = \{M\in \mathcal{A}~ |~\Ext_{\mathcal{A}}^{1}(X, M) = 0, \forall X\in\mathcal{X}\}$
and ${^{\perp}\mathcal{Y}}$ is defined dually.
The \emph{kernel} of a cotorsion pair $(\mathcal{X}, \mathcal{Y})$
is defined as $\mathcal{X} \cap \mathcal{Y}$.
A cotorsion pair $(\mathcal{X}, \mathcal{Y})$ is said to be
\emph{generated by a set} $\mathcal{S}$
if $\mathcal{S}^\perp = \mathcal{Y}$.

A cotorsion pair $(\mathcal{X}, \mathcal{Y})$ is said to be
\emph{complete} if for every $M \in \mathcal{A}$,
there are conflations
$Y \rightarrowtail X \twoheadrightarrow M$
and
$M \rightarrowtail Y^{\prime} \twoheadrightarrow X^{\prime}$,
where $X$, $X^{\prime} \in \mathcal{X}$ and
$Y$, $Y^{\prime} \in \mathcal{Y}$.
We note that the existence of such conflations
($\forall M \in \mathcal{A}$) are equivalent
in case $\mathcal{A}$ has enough projectives and injectives
by a similar argument to \cite[Proposition 7.1.7]{EJ}.

A cotorsion pair $(\mathcal{X}, \mathcal{Y})$ is said to be \emph{hereditary}
if (1) $\mathcal{X}$ is closed under kernels of deflations
and (2) $\mathcal{Y}$ is closed under cokernels of inflations.
In particular, for a complete cotorsion pair $(\mathcal{X}, \mathcal{Y})$ in a WIC exact category,
the above two conditions (1) and (2) are equivalent,
and moreover,
they are also equivalent to $\Ext_{\mathcal{A}}^n(X, Y) = 0$
for each $X \in \mathcal{X}, Y \in \mathcal{Y}$ and $n \geq 1$ (see \cite[Lemma 6.17]{Sto201401}).

\subsection{\bf Dimensions }

Let $\mathcal{A}$ be an exact category and
$(\mathcal{X}, \mathcal{Y})$ a cotorsion pair.
Given $M \in \mathcal{A}$,
the $\mathcal{X}$-\emph{projective dimension} of $M$,
denoted by $\mathcal{X}\text{-}\pd_{\mathcal{A}}(M)$,
is defined as follows:

$$\mathcal{X}\text{-}\pd_{\mathcal{A}}(M)
= \inf \left \{n\in \mathbb{N} \hspace{0.2cm}
\begin{array}{ |l} \text{there is an acyclic sequence} \\
	X_n \rightarrowtail X_{n-1} \to \cdots \to X_1 \to X_0 \twoheadrightarrow M \\
	\text{in } \mathcal{A} \text{ with each } X_i\in\mathcal{X}\end{array}\right\}.
$$
If no such an acyclic sequence exists,
then we set $\mathcal{X}\text{-}\pd_\mathcal{A}(M) = \infty$.
Dually, we have the definition of
$\mathcal{Y}$-\emph{injective dimension} of $M$,
denoted by $\mathcal{Y}\text{-}\id_\mathcal{A}(M)$.

The proofs of the following two lemmas are standard and we omit the details.

\begin{lem} \label{06-1}
	Let $\mathcal{A}$ be a WIC exact category with enough projectives and injectives,
	and $(\mathcal{X}, \mathcal{X}^{\perp})$ a complete and hereditary cotorsion pair in $\mathcal{A}$.
	Then the following are equivalent for any $M\in \mathcal{A}$
	and any integer $n\geq0$:

    \emph{(1)} $\mathcal{X}\text{-}\pd_{\mathcal{A}}(M)\leq n$.
	
	\emph{(2)}	There is an acyclic sequence $ X_n\rightarrowtail \cdots\to X_1\to X_0\twoheadrightarrow M $ with each $X_i\in \mathcal{X}$.
	
	\emph{(3)}	 For any acyclic sequence $ K\rightarrowtail X_{n-1}\to\cdots\to X_1\to X_0\twoheadrightarrow M $, if each $X_i$ is in $\mathcal{X}$, then so is $K$.
		
	\emph{(4)}	 $\Ext_\mathcal{A}^{n+i}(M,Y)=0$ for all $Y \in \mathcal{X}^{\perp}$ and $i \geq 1$.
\end{lem}

\begin{lem} \label{06-2}
	Let $\mathcal{A}$ be a WIC exact category with enough projectives and injectives,
	and $(\mathcal{X}, \mathcal{X}^{\perp})$ a complete and hereditary cotorsion pair in $\mathcal{A}$.
	Then the following hold for any conflation $A \rightarrowtail B \twoheadrightarrow C$
	in $\mathcal{A}$:
	
	\emph{(1)}	$\mathcal{X}\text{-}\pd_\mathcal{A}(A)\leq \max\{\mathcal{X}\text{-}\pd_\mathcal{A}(B), \mathcal{X}\text{-}\pd_\mathcal{A}(C)-1\}$ with the equality if $\mathcal{X}\text{-}\pd_\mathcal{A}(B)$ $\neq \mathcal{X}\text{-}\pd_\mathcal{A}(C)$. If $\mathcal{X}\text{-}\pd_\mathcal{A}(C)=0$, then $\mathcal{X}\text{-}\pd_\mathcal{A}(C)-1$ should be interpreted as 0.
	
	\emph{(2)}	$\mathcal{X}\text{-}\pd_\mathcal{A}(B)\leq \max\{\mathcal{X}\text{-}\pd_\mathcal{A}(A), \mathcal{X}\text{-}\pd_\mathcal{A}(C)\}$ with the equality if $\mathcal{X}\text{-}\pd_\mathcal{A}(C)\neq \mathcal{X}\text{-}\pd_\mathcal{A}(A)+1$.
	
	\emph{(3)}	$\mathcal{X}\text{-}\pd_\mathcal{A}(C)\leq \max\{\mathcal{X}\text{-}\pd_\mathcal{A}(B), \mathcal{X}\text{-}\pd_\mathcal{A}(A)+1\}$ with the equality if $\mathcal{X}\text{-}\pd_\mathcal{A}(B)$ $\neq \mathcal{X}\text{-}\pd_\mathcal{A}(A)$.
\end{lem}

\subsection{\bf Model structures }

We refer the reader to \cite{Quillen1967} and \cite{Hov1999}
for some basic concepts and results on model structure and model category.

According to \cite[Definition 3.1]{Gil11},
an \emph{exact model structure} in an exact category $\mathcal{A}$
is a model structure such that
(trivial) cofibrations are exactly inflations
with a (trivially) cofibrant cokernel,
and (trivial) fibrations are exactly deflations
with a (trivially) fibrant kernel.
The concept of \emph{abelian model structure} in \cite{Hov2002}
is a special case of that of exact model structure.

Let $\mathcal{A}$ be an exact category,
a triple ($\mathcal{C}$, $\mathcal{W}$, $\mathcal{F}$)
of classes of objects of $\mathcal{A}$ is called a \emph{Hovey triple}
if ($\mathcal{C} \cap \mathcal{W}$, $\mathcal{F}$)
and ($\mathcal{C}$, $\mathcal{W} \cap \mathcal{F}$)
are complete cotorsion pairs,
and $\mathcal{W}$ is \emph{thick} in $\mathcal{A}$
(i.e. $\mathcal{W}$ is closed under direct summands,
and if two out of three objects in a conflation are in $\mathcal{W}$,
then so is the third one).
A Hovey triple is called \emph{hereditary}
if these two cotorsion pairs
($\mathcal{C} \cap \mathcal{W}$, $\mathcal{F}$)
and ($\mathcal{C}$, $\mathcal{W} \cap \mathcal{F}$) are hereditary.

Hovey \cite{Hov2002} showed that there is a translation between
abelian model structures and Hovey triples in an abelian category.
Gillespie \cite{Gil11} extended this correspondence
to WIC exact categories.
Let $\mathcal{A}$ be a WIC exact category,
there is a one-to-one correspondence between
exact model structures and Hovey triples,
thus we may consider Hovey triples and exact model structures
to be the same when there is no risk of confusion.

Given a Frobenius category $\mathcal{B}$
with the class of projective-injective objects $\mathcal{P}$,
Happel \cite[Theorem 2.6]{H1988} proved that the stable category
$\underline{\mathcal{B}}=\mathcal{B}/\mathcal{P}$ is a triangulated category.
For a WIC exact category $\mathcal{A}$,
the homotopy category induced by a Hovey triple
$\mathcal{M} = (\mathcal{C}, \mathcal{W}, \mathcal{F})$ in $\mathcal{A}$
is denoted by $\Ho(\mathcal{M})$.
If $\mathcal{M}$ is hereditary, by virtue of \cite[Theorem 6.21]{Sto201401},
$\mathcal{C}\cap \mathcal{F}$ is a Frobenius category
whose class of projective-injective objects is $\mathcal{C}\cap \mathcal{W} \cap \mathcal{F}$,
and there is a triangle equivalence
$$\Ho(\mathcal{M})\simeq \underline{\mathcal{C}\cap \mathcal{F}}=(\mathcal{C}\cap \mathcal{F})/(\mathcal{C}\cap \mathcal{W}\cap \mathcal{F}).$$

\section{\bf Main results}

Let $\mathcal{A}$ be an exact category.
Given an integer $n\geq 0$ and a cotorsion pair
$(\mathcal{X}, \mathcal{Y})$,
we denote by $\mathcal{X}_{n}$
(resp., $\mathcal{Y}_{n}$) the subclass of $\mathcal{A}$
consisting of all objects with $\mathcal{X}$-projective
(resp., $\mathcal{Y}$-injective) dimension at most $n$.
For the sake of brevity,
we write $\mathcal{X}_n^{\perp}$ and $^{\perp}\mathcal{Y}_n$
instead of $(\mathcal{X}_n)^{\perp}$ and $^{\perp}(\mathcal{Y}_n)$,
respectively.

\begin{df} \label{D-1}
	{\rm Let $\mathcal{A}$ be an exact category and
		$(\mathcal{X}, \mathcal{Y})$ a complete cotorsion pair in $\mathcal{A}$.
		We say that $(\mathcal{X}, \mathcal{Y})$ is \emph{left}
		(resp., \emph{right}) \emph{extendable} if the pair $(\mathcal{X}_n, \mathcal{X}_n^{\perp})$
		(resp., $(^{\perp}\mathcal{Y}_n, \mathcal{Y}_n)$) is also a complete cotorsion pair for any integer $n\geq 0$.
	}
\end{df}

\begin{exa} \label{E}
	{\rm Let $R$ be a ring.
		Consider the category $R$-Mod of all left $R$-modules.
		Let $\mathcal{P}(R)$ (resp., $\mathcal{I}(R)$ and $\mathcal{F}(R)$)
		be the class of all projective (resp., injective and flat)
		left $R$-modules.
		It turns out that the cotorsion pairs
		$(\mathcal{P}(R), R\text{-}\Mod)$ and
		$(\mathcal{F}(R), \mathcal{F}(R)^{\perp})$ are left extendable,
		and $(R\text{-}\Mod, \mathcal{I}(R))$ is right extendable
		(see, e.g., \cite[Theorems 4.1.3, 4.1.7 and 4.1.12]{GT2006}).
	}
\end{exa}

\begin{rem} \label{T-2-1}
	{\rm If a left (resp., right) extendable cotorsion pair $(\mathcal{X}, \mathcal{Y})$
		in a WIC exact category with enough projectives and injectives is hereditary,
		then, for any integer $n\geq 0$,
		the complete cotorsion pair
		$(\mathcal{X}_n, \mathcal{X}_n^{\perp})$
		(resp., $(^{\perp}\mathcal{Y}_n, \mathcal{Y}_n)$)
		is hereditary as well by Lemma \ref{06-2} and \cite[Lemma 6.17]{Sto201401}.
	}
\end{rem}

Inspired by the spirit of \cite[Theorem 4.11]{SS2020},
we obtain the following result, which is of interest in its own right, will be used to prove our main Theorem \ref{modle structure}.

\begin{thm}\label{SStype}
	Let $\mathcal{A}$ be a WIC exact category with enough projectives and injectives,
	and $(\mathcal{C}, \mathcal{W}, \mathcal{F})$ a hereditary Hovey triple in $\mathcal{A}$.
	If the cotorsion pair $(\mathcal{C}\cap \mathcal{W}, \mathcal{F})$ is left extendable,
	then the following assertions are equivalent for any $M \in \mathcal{A}$ and any non-negative integer $n$.
	
	\emph{(1)} $M \in \mathcal{C}_{n}$.
	
	\emph{(2)} There is a conflation $K \rightarrowtail G \twoheadrightarrow M$ in $\mathcal{A}$
	with $G \in \mathcal{C}$ and
	$K \in (\mathcal{C}\cap \mathcal{W})_{n-1}$
	such that
	$$0 \to \Hom_\mathcal{A}(M, Y) \to \Hom_\mathcal{A}(G, Y) \to \Hom_\mathcal{A}(K, Y) \to  0$$
	is exact for all $Y \in (\mathcal{C}\cap \mathcal{W})_{n}^{\perp}$
	\emph{(}If $n=0$, then $(\mathcal{C}\cap \mathcal{W})_{n-1}$ should be interpreted as $\{0\}$\emph{)}.
	
	\emph{(3)} $\Ext_\mathcal{A} ^{1}(M, H) = 0$ for every $H \in \mathcal{W} \cap (\mathcal{C}\cap \mathcal{W})_{n}^{\perp}$.
	
	\emph{(4)} There is a conflation $M \rightarrowtail L \twoheadrightarrow N $ in $\mathcal{A}$
	with $L \in (\mathcal{C}\cap \mathcal{W})_{n}$ and $N \in \mathcal{C}$.
\end{thm}
\begin{proof}
	The arguments in (1)$\Leftrightarrow$(4) and the first part of (1)$\Rightarrow$(2) can be traced back to Auslander-Buchweitz approximation theory (see for instance \cite{AB1989}).
	
	(4) $\Rightarrow$ (1).
	Assume that there is a conflation $ M \rightarrowtail L \twoheadrightarrow N$ in $\mathcal{A}$
	with $L \in (\mathcal{C}\cap \mathcal{W})_{n}$ and $N \in \mathcal{C}$.
	Since $(\mathcal{C}\cap \mathcal{W})_{n} \subseteq \mathcal{C}_{n}$,
	we have $M \in \mathcal{C}_{n}$ by Lemma \ref{06-2}(1).
	
	(1) $\Rightarrow$ (4).
	Assume that $M \in \mathcal{C}_{n}$.
	We will proceed by induction on $n$.
	For $n = 0$, there is a conflation $$M \rightarrowtail L \twoheadrightarrow N$$
	with $L \in \mathcal{W}\cap \mathcal{F}$ and $N \in \mathcal{C}$
	since $(\mathcal{C}, ~\mathcal{W}\cap \mathcal{F})$ is a complete cotorsion pair.
	Notice further that
	$\mathcal{C}$ is closed under extensions and $M, N \in \mathcal{C}$,
	so we have
	$L \in \mathcal{C}\cap \mathcal{W} \cap \mathcal{F} \subseteq \mathcal{C}\cap \mathcal{W}$.
	
	Now assume $n \ge 1$.
	In view of the completeness of the cotorsion pair
	$(\mathcal{C}, ~\mathcal{W}\cap \mathcal{F})$,
	there is a conflation $K \rightarrowtail C_{0} \twoheadrightarrow M$
	with $K \in \mathcal{W}\cap \mathcal{F}$ and $C_{0} \in \mathcal{C}$.
	By Lemma \ref{06-2}(1) it follows that $K \in \mathcal{C}_{n-1}$.
	Hence, the induction hypothesis yields a conflation
	$$K \rightarrowtail L_{1} \twoheadrightarrow N_{1}$$
	with $L_{1} \in (\mathcal{C} \cap \mathcal{W})_{n-1}$ and
	$N_{1} \in \mathcal{C}$.
	Consider the following pushout diagram:
	\begin{displaymath}
		\xymatrix{
			K \ar@{>->}[r] \ar@{>->}[d] & C_{0} \ar@{->>}[r] \ar@{>-->}[d] & M \ar@{=}[d] \\
			L_{1} \ar@{>-->}[r] \ar@{->>}[d] & D \ar@{->>}[r] \ar@{->>}[d] & M  \\
			N_{1} \ar@{=}[r]  & N_{1} &
		}
	\end{displaymath}
	where $D \in \mathcal{C}$ since $C_{0}, N_{1} \in \mathcal{C}$.
	By the case where $n = 0$, there is a conflation
	$$ D \rightarrowtail L_{2} \twoheadrightarrow N $$
	with $L_{2} \in \mathcal{C}\cap \mathcal{W}$ and $N \in \mathcal{C}$.
	Now we have another pushout diagram
	\begin{displaymath}
		\xymatrix{
			L_{1} \ar@{>->}[r] \ar@{=}[d] & D \ar@{->>}[r] \ar@{>->}[d] & M \ar@{>-->}[d] \\
			L_{1} \ar@{>->}[r] & L_{2} \ar@{-->>}[r] \ar@{->>}[d] & L \ar@{->>}[d] \\
			& N \ar@{=}[r]  & N
		}
	\end{displaymath}
	Applying Lemma \ref{06-2}(3) to the conflation $L_1 \rightarrowtail L_2 \twoheadrightarrow L$,
	one can obtain $L\in (\mathcal{C}\cap \mathcal{W})_n$.
	Thus the right column in the above diagram is as desired.
	
	(1) $\Rightarrow$ (2).
	Assume that $M \in \mathcal{C}_{n}$.
	The case when $n = 0$ is clear.
	Now suppose $n \ge 1$.
	Since $(\mathcal{C}, ~\mathcal{W}\cap \mathcal{F})$
	is a complete cotorsion pair,
	we can take a conflation
	$$(\dag)\quad K \rightarrowtail G \twoheadrightarrow M$$
	with $K \in \mathcal{W}\cap \mathcal{F}$ and $G \in \mathcal{C}$.
	By Lemma \ref{06-2}(1), one can see that
	$K \in \mathcal{C}_{n-1}\cap \mathcal{W}\cap \mathcal{F}$.
	Now by (4) (noting that (1)$\Rightarrow$(4) is proved above),
	there exists a conflation
	$$(\ddag)\quad K \rightarrowtail L \twoheadrightarrow N $$
	with $L \in (\mathcal{C}\cap \mathcal{W})_{n-1}$ and $N \in \mathcal{C}$.
	Since $\Ext_\mathcal{A} ^{1}(N, K) = 0$,
	it follows that the conflation $(\ddag)$ splits
	and hence $K \in (\mathcal{C}\cap \mathcal{W})_{n-1}$.
	
	By a similar argument to the case when $n = 0$
	in the proof of (1) $\Rightarrow$ (4),
	one can take a conflation
	$$ G \rightarrowtail S \twoheadrightarrow T $$
	with $S \in \mathcal{C}\cap \mathcal{W}\cap \mathcal{F}$
	and $T \in \mathcal{C}$.
	Now, consider the pushout diagram:
	\begin{displaymath}
		\xymatrix{
			K \ar@{>->}[r] \ar@{=}[d] & G \ar[r] \ar@{>->}[d] & M \ar@{>-->}[d]  \\
			K \ar@{>->}[r] & S \ar@{-->>}[r] \ar@{->>}[d] & B \ar@{->>}[d] \\
			& T \ar@{=}[r] & T
		}
	\end{displaymath}
	One has $B \in (\mathcal{C}\cap \mathcal{W})_{n}$ by Lemma \ref{06-2}(3),
	and hence $\Ext_\mathcal{A} ^{1}(B, Y) = 0$
	for any $Y \in (\mathcal{C}\cap \mathcal{W})_{n}^{\perp}$.
	Furthermore, we have the following commutative diagram with exact rows
	\begin{displaymath}
		\xymatrix{
			0\ar[r] & \Hom_\mathcal{A}(B, Y)\ar[r] \ar[d] & \Hom_\mathcal{A}(S, Y)\ar[r] \ar[d] & \Hom_\mathcal{A}(K, Y)\ar[r] \ar@{=}[d] & \Ext_\mathcal{A} ^{1}(B, Y) = 0 \\
			0\ar[r] & \Hom_\mathcal{A}(M, Y)\ar[r]        & \Hom_\mathcal{A}(G, Y)\ar[r]^{f}    & \operatorname{Hom}_\mathcal{A}(K, Y)  	}
	\end{displaymath}
	It is easy to see that $f: \operatorname{Hom}_\mathcal{A}(G, Y) \to \operatorname{Hom}_\mathcal{A}(K, Y)$ is surjective.
	Therefore, (2) follows.		
	
	(2) $\Rightarrow$ (3).
	The case when $n=0$ is trivial.
	Now let $n > 0$ and assume that there is a conflation
	$K \rightarrowtail G \twoheadrightarrow M$ in $\mathcal{A}$
	with $G \in \mathcal{C}$ and
	$K \in (\mathcal{C}\cap \mathcal{W})_{n-1}$
	such that
	$$0 \to \Hom_\mathcal{A}(M, Y) \to \Hom_\mathcal{A}(G, Y) \to \Hom_\mathcal{A}(K, Y) \to  0$$
	is exact for all $Y \in (\mathcal{C}\cap \mathcal{W})_{n}^{\perp}$.
	Now (3) holds in view of the following exact sequence
	$$\Hom_\mathcal{A}(G, H) \overset{g}{\longrightarrow} \Hom_\mathcal{A}(K, H) \to
	\Ext_\mathcal{A} ^{1}(M, H) \to \Ext_\mathcal{A} ^{1}(G, H)$$
	where $g$ is surjective and \Ext$_\mathcal{A} ^{1}(G, H) = 0$
	for all $H \in ~\mathcal{W}\cap (\mathcal{C}\cap \mathcal{W})_{n}^{\perp} \subseteq{\mathcal{W} \cap \mathcal{F}} = \mathcal{C}^{\perp}$.
	
	(3) $\Rightarrow$ (4).
	Assume that $\Ext_\mathcal{A} ^{1}(M, H) = 0$ for every
	$H \in \mathcal{W} \cap (\mathcal{C}\cap \mathcal{W})_{n}^{\perp}$.
	Since the cotorsion pair $(\mathcal{C}, \mathcal{W}\cap \mathcal{F})$
	is complete, there is a conflation
	$$ M \rightarrowtail L \twoheadrightarrow N $$
	with $L \in \mathcal{W}\cap \mathcal{F}$ and $N \in \mathcal{C}$.
	We have to show that $L \in (\mathcal{C}\cap \mathcal{W})_{n}$.
	Indeed, for every
	$H \in \mathcal{W} \cap (\mathcal{C}\cap \mathcal{W})_{n}^{\perp}$,
	$0 = \Ext_\mathcal{A} ^{1}(N, H) \to \Ext_\mathcal{A} ^{1}(L, H) \to \Ext_\mathcal{A} ^{1}(M, H) = 0$
	yields $\Ext_\mathcal{A} ^{1}(L, H) = 0$.
	This shows $L \in {}^{\perp}(\mathcal{W} \cap (\mathcal{C}\cap \mathcal{W})_{n}^{\perp})$.
	Moreover,
	since the cotorsion pair $(\mathcal{C}\cap \mathcal{W}, \mathcal{F})$
	is left extendable,
	there is a complete cotorsion pair
	$((\mathcal{C}\cap \mathcal{W})_{n}, (\mathcal{C}\cap \mathcal{W})_{n}^{\perp})$.
	Consequently, we can obtain a conflation
	$$(\sharp)\quad  K \rightarrowtail L_{1} \twoheadrightarrow L $$
	with $L_{1} \in (\mathcal{C}\cap \mathcal{W})_{n}$ and
	$K \in (\mathcal{C}\cap \mathcal{W})_{n}^{\perp}$.
	Take an acyclic sequence
	$$X_n \rightarrowtail X_{n-1} \to \cdots \to X_1 \to X_0 \twoheadrightarrow L_1 $$
	with each $X_i \in \mathcal{C}\cap \mathcal{W}$.
	Since $\mathcal{W}$ is thick,
	it is easy to see that $L_1 \in \mathcal{W}$.
	Simultaneously, we have $L \in \mathcal{W}\cap \mathcal{F}\subseteq \mathcal{W}$.
	Thus the thickness of $\mathcal{W}$ guarantees $K\in \mathcal{W}$.
	As we have shown $L \in {}^{\perp}(\mathcal{W} \cap (\mathcal{C}\cap \mathcal{W})_{n}^{\perp})$,
	it follows that $\Ext_\mathcal{A} ^{1}(L, K) = 0$,
	which means the conflation $(\sharp)$ is split.
	Therefore $L \in (\mathcal{C}\cap \mathcal{W})_{n}$ as desired.
\end{proof}

We remark that ``$Y \in (\mathcal{C}\cap \mathcal{W})_{n}^{\perp}$''
in condition (2) of the Theorem \ref{SStype} can be
replaced by ``$Y \in \mathcal{W}\cap(\mathcal{C}\cap \mathcal{W})_{n}^{\perp}$''.

\begin{cor}\label{SStype-0}
	Let $\mathcal{A}$ be a WIC exact category with enough projectives and injectives,
	and $(\mathcal{C}, \mathcal{W}, \mathcal{F})$ a hereditary Hovey triple in $\mathcal{A}$.
	Then $\mathcal{C}_{n} \cap \mathcal{W} = (\mathcal{C}\cap \mathcal{W})_{n}$
	for any non-negative integer $n$.		
\end{cor}
\begin{proof}
	Clearly, $(\mathcal{C}\cap \mathcal{W})_{n} \subseteq \mathcal{C}_{n}$.
	In addition, $(\mathcal{C}\cap \mathcal{W})_{n} \subseteq \mathcal{W}$ as $\mathcal{W}$ is thick.
	Thus, $(\mathcal{C}\cap \mathcal{W})_{n} \subseteq \mathcal{C}_{n} \cap \mathcal{W}$.
	Conversely, if $M \in \mathcal{C}_{n} \cap \mathcal{W}$,
	then by the proof of (1) $\Rightarrow$ (2) of Theorem \ref{SStype},
	there is a conflation
	$$ K \rightarrowtail G \twoheadrightarrow M $$
	with $G \in \mathcal{C}$ and
	$K \in (\mathcal{C}\cap \mathcal{W})_{n-1} \subseteq \mathcal{W}$.
	On the other hand, since $\mathcal{W}$ is thick,
	one can see that $G\in \mathcal{W}$.
	Now $M \in (\mathcal{C}\cap \mathcal{W})_{n}$ by Lemma \ref{06-2}(3).
	Therefore, $\mathcal{C}_{n} \cap \mathcal{W} = (\mathcal{C}\cap \mathcal{W})_{n}$.
\end{proof}

We are now in a position to prove our main result.

\begin{thm}\label{modle structure}
	Let $\mathcal{A}$ be a WIC exact category with enough projectives and injectives, and $\mathcal{M}=(\mathcal{C}, \mathcal{W}, \mathcal{F})$ a hereditary Hovey triple in $\mathcal{A}$.
	If the cotorsion pair $(\mathcal{C}\cap \mathcal{W}, \mathcal{F})$ is left extendable, then the following hold for any  integer $n\geq 0$.
	
	\emph{(1)} The cotorsion pair $(\mathcal{C}, \mathcal{W}\cap \mathcal{F})$ is also left extendable.
	More precisely,
	$$(\mathcal{C}_{n}, \mathcal{C}_{n}^{\perp})
	= (\mathcal{C}_{n}, \mathcal{W} \cap (\mathcal{C}\cap \mathcal{W})_{n}^{\perp})$$
	is a complete and hereditary cotorsion pair with kernel $(\mathcal{C}\cap \mathcal{W})_{n} \cap (\mathcal{C}\cap \mathcal{W})_{n}^{\perp}$.
	
	\emph{(2)} The triple $\mathcal{M}_{n} =(\mathcal{C}_{n}, \mathcal{W}, (\mathcal{C}\cap \mathcal{W})_{n} ^{\perp})$ is a hereditary Hovey triple in $\mathcal{A}$.
	Moreover, $\mathcal{C}_{n} \cap (\mathcal{C}\cap \mathcal{W})_{n} ^{\perp}$ is a Frobenius category with the class of projective-injective objects $(\mathcal{C}\cap \mathcal{W})_{n} \cap (\mathcal{C}\cap \mathcal{W})_{n} ^{\perp}$.
	Furthermore, there are triangle equivalences
	$$\underline{\mathcal{C}_{n} \cap (\mathcal{C}\cap \mathcal{W})_{n} ^{\perp}} \simeq \Ho(\mathcal{M}_{n})= \Ho(\mathcal{M})\simeq \underline{\mathcal{C}\cap \mathcal{F}}.$$
\end{thm}
\begin{proof}
	The case $n=0$ is obvious.
	Now assume $n> 0$.
	
	(1) By Theorem \ref{SStype},
	we have $\mathcal{C}_n = {}^{\perp}(\mathcal{W} \cap (\mathcal{C}\cap \mathcal{W})_{n} ^{\perp})$
	and $\mathcal{W} \cap (\mathcal{C}\cap \mathcal{W})_{n} ^{\perp} \subseteq \mathcal{C}_n^{\bot}$.
	In addition,
	$$\mathcal{C}_n^{\perp}  \subseteq ((\mathcal{C}\cap \mathcal{W})_{n} \cup \mathcal{C})^{\perp} = (\mathcal{C}\cap \mathcal{W})_{n} ^{\perp} \cap \mathcal{C} ^{\perp} = (\mathcal{C}\cap \mathcal{W})_{n} ^{\perp} \cap \mathcal{W} \cap \mathcal{F}  \subseteq \mathcal{W} \cap (\mathcal{C}\cap \mathcal{W})_{n} ^{\perp}.$$
	Thus, $(\mathcal{C}_n, \mathcal{W}\cap (\mathcal{C}\cap \mathcal{W})_{n} ^{\perp})$ is a cotorsion pair.
	By Corollary \ref{SStype-0}, $\mathcal{W} \cap \mathcal{C}_{n} = (\mathcal{C} \cap \mathcal{W})_{n}$,
	then the kernel of $(\mathcal{C}_{n}, \mathcal{W} \cap (\mathcal{C}\cap \mathcal{W})_{n}^{\perp})$ is $(\mathcal{C} \cap \mathcal{W})_{n} \cap (\mathcal{C} \cap \mathcal{W})_{n} ^{\perp}$.
	
	Next, we show that the cotorsion pair
	$(\mathcal{C}_n, \mathcal{W}\cap (\mathcal{C}\cap \mathcal{W})_{n} ^{\perp})$
	is complete.
	It suffices to construct a conflation
	$$(\dag)\quad L_1 \rightarrowtail G_1 \twoheadrightarrow M $$
	with $G_1 \in \mathcal{C}_n$ and $L_1 \in \mathcal{W}\cap (\mathcal{C}\cap \mathcal{W})_{n} ^{\perp}$,
	for any $M \in \mathcal{A}$.
	To this end, by the completeness of the cotorsion pair $(\mathcal{C}, \mathcal{W}\cap\mathcal{F})$,
	we can take a conflation
	$$ N \rightarrowtail H \twoheadrightarrow M $$
	with $H \in \mathcal{C}$ and $N \in \mathcal{W}\cap \mathcal{F}$.
	Since the cotorsion pair $(\mathcal{C} \cap \mathcal{W}, \mathcal{F})$
	is left extendable,
	we have a complete cotorsion pair
	$((\mathcal{C} \cap \mathcal{W})_{n}, (\mathcal{C} \cap \mathcal{W})_{n}^{\perp})$,
	from which one can obtain a conflation
	$$ N \rightarrowtail L_1 \twoheadrightarrow B $$
	with $L_1 \in (\mathcal{C} \cap \mathcal{W})_{n} ^{\perp}$
	and $B \in (\mathcal{C} \cap \mathcal{W})_{n}$.
	Consider the following pushout diagram
	\begin{displaymath}
		\xymatrix{
			N \ar@{>->}[r] \ar@{>->}[d] & H \ar@{->>}[r] \ar@{>-->}[d] & M  \ar@{=}[d] \\
			L_1 \ar@{>-->}[r] \ar@{->>}[d] & G_1 \ar@{->>}[r] \ar@{->>}[d] & M  \\
			B \ar@{=}[r]  & B  &
		}
	\end{displaymath}
	Since $H \in \mathcal{C}$ and $B \in (\mathcal{C} \cap \mathcal{W})_{n} \subseteq \mathcal{C}_{n}$,
	one can see that $G_1 \in \mathcal{C}_{n}$ by Lemma \ref{06-2}(2).
	On the other hand,
	as we have seen that $N \in \mathcal{W}\cap \mathcal{F} \subseteq \mathcal{W}$
	and $B \in (\mathcal{C} \cap \mathcal{W})_{n}
	\overset{\mathrm{Corollary\ \ref{SStype-0}}}{=\!\!\!=\!\!\!=\!\!\!=\!\!\!=\!\!\!=\!\!\!=\!\!\!=\!\!\!=\!\!\!=\!\!\!=\!\!\!=\!\!\!=\!\!\!=}
	\mathcal{W}\cap \mathcal{C}_{n} \subseteq \mathcal{W}$,
	it follows that $L_1\in \mathcal{W}$ since $\mathcal{W}$ is thick.
	Thus the conflation $L_1 \rightarrowtail G_1 \twoheadrightarrow M$
	in the above diagram is as desired.
	
	Finally, by Remark \ref{T-2-1},
	$(\mathcal{C}_n, \mathcal{W}\cap (\mathcal{C}\cap \mathcal{W})_{n} ^{\perp})$ is also hereditary.
	
	(2) Note that $\mathcal{W}$ is thick and the complete cotorsion pair $((\mathcal{C} \cap \mathcal{W})_{n}, (\mathcal{C} \cap \mathcal{W})_{n}^{\perp})$
	is hereditary by Remark \ref{T-2-1}.
	Combining this cotorsion pair with $(\mathcal{C}_{n}, \mathcal{W} \cap (\mathcal{C}\cap \mathcal{W})_{n}^{\perp})$
	in (1), we obtain a hereditary Hovey triple
	$\mathcal{M}_{n} = (\mathcal{C}_{n}, \mathcal{W}, (\mathcal{C} \cap \mathcal{W})_{n} ^{\perp})$.
	
	The remaining assertions hold by \cite[Theorem 6.21]{Sto201401} and \cite[Corollary 1.2]{GZ2024}.
\end{proof}

The next three results are dual to Theorem \ref{SStype}, Corollary \ref{SStype-0} and Theorem \ref{modle structure}, respectively.

\begin{thm}\label{SStype_{1}}
	Let $\mathcal{A}$ be a WIC exact category with enough projectives and injectives,
	and $(\mathcal{C}, \mathcal{W}, \mathcal{F})$ a hereditary Hovey triple in $\mathcal{A}$.
	If the cotorsion pair $(\mathcal{C}, \mathcal{W}\cap \mathcal{F})$ is right extendable, then the following assertions are equivalent for any $M \in \mathcal{A}$ and any integer $n\geq 0$.
	
	\emph{(1)} $M \in \mathcal{F}_{n}$.
	
	\emph{(2)} There is a conflation $M \rightarrowtail G \twoheadrightarrow C$ in $\mathcal{A}$
	with $G \in \mathcal{F}$
	and $C \in (\mathcal{W}\cap \mathcal{F})_{n-1}$
	such that
	$$0 \to \Hom_\mathcal{A}(X, M) \to \Hom_\mathcal{A}(X, G) \to \Hom_\mathcal{A}(X, C) \to  0$$
	is exact for all $X \in {}^\perp(\mathcal{W}\cap \mathcal{F})_{n}$
	\emph{(}If $n=0$, then $(\mathcal{W}\cap \mathcal{F})_{n-1}$
	should be interpreted as $\{0\}$\emph{)}.
	
	\emph{(3)} $\Ext_\mathcal{A}^{1}(H, M) = 0$ for every $H \in \mathcal{W} \cap {}^\perp(\mathcal{W}\cap \mathcal{F})_{n}$.
	
	\emph{(4)} There is a conflation $ N \rightarrowtail L \twoheadrightarrow M $ in $\mathcal{A}$
	with $L \in (\mathcal{W}\cap \mathcal{F})_{n}$ and $N \in \mathcal{F}$.
\end{thm}

\begin{cor}\label{SStype-0-1}
	Let $\mathcal{A}$ be a WIC exact category with enough projectives and injectives,
	and $(\mathcal{C}, \mathcal{W}, \mathcal{F})$ a hereditary Hovey triple in $\mathcal{A}$.
	Then $\mathcal{W} \cap \mathcal{F}_{n} = (\mathcal{W}\cap \mathcal{F})_{n}$ for any non-negative integer $n$.		
\end{cor}

\begin{thm}\label{modle structure 1}
	Let $\mathcal{A}$ be a WIC exact category with enough projectives and injectives, and $\mathcal{M}=(\mathcal{C}, \mathcal{W}, \mathcal{F})$ a hereditary Hovey triple in $\mathcal{A}$.
	If the cotorsion pair $(\mathcal{C}, \mathcal{W}\cap \mathcal{F})$ is right extendable, then the following hold for any integer $n\geq0$.
	
	\emph{(1)} The cotorsion pair $(\mathcal{C}\cap \mathcal{W}, \mathcal{F})$ is also right extendable.
	More precisely,
	$$({}^{\perp}\mathcal{F}_{n}, \mathcal{F}_{n})
	= (^{\perp}(\mathcal{W}\cap\mathcal{F})_{n} \cap \mathcal{W}, \mathcal{F}_{n})$$
	is a complete and hereditary cotorsion pair with kernel
	$^{\perp}(\mathcal{W}\cap\mathcal{F}) _{n} \cap (\mathcal{W}\cap\mathcal{F})_{n}$.
	
	\emph{(2)} The triple $\mathcal{M}_n =(^{\perp}(\mathcal{W}\cap\mathcal{F}) _{n}, \mathcal{W}, \mathcal{F}_{n})$ is a hereditary Hovey triple in $\mathcal{A}$.
	Moreover, $^{\perp}(\mathcal{W}\cap\mathcal{F})_{n} \cap \mathcal{F}_{n}$
	is a Frobenius category with the class of projective-injective objects
	$^{\perp}(\mathcal{W}\cap \mathcal{F})_{n} \cap (\mathcal{W}\cap \mathcal{F})_{n}$.
	Furthermore, there are triangle equivalences
	$$\underline{^{\perp}(\mathcal{W}\cap \mathcal{F})_{n} \cap \mathcal{F}_{n}} \simeq \Ho(\mathcal{M}_{n})= \Ho(\mathcal{M})\simeq \underline{\mathcal{C}\cap \mathcal{F}}.$$
\end{thm}

\begin{rem}
\emph{After an earlier version of this paper was finished,
we note that El Maaouy also obtained the triple $\mathcal{M}_n =(^{\perp}(\mathcal{W}\cap\mathcal{F}) _{n}, \mathcal{W}, \mathcal{F}_{n})$
independently under deferent assumptions.
See \cite[Theorem A]{EL25} for more details.}
\end{rem}

We have been unable to obtain an equivalent condition
for a cotorsion pair $(\mathcal{X}, \mathcal{Y})$
to be left (resp., right) extendable.
In the remaining part of this section,
we will give a sufficient condition under which a cotorsion pair
$(\mathcal{X}, \mathcal{Y})$ is right extendable.

Let $\lambda$ be a regular cardinal.
Recall that a poset $I$ is said to be $\lambda$-\emph{directed}
if each subset of $I$ of cardinality smaller than $\lambda$ has an upper bound.
Let $\mathcal{G}$ be an additive category with arbitrary $\lambda$-direct limits
(i.e., colimits of all diagrams whose shapes are $\lambda$-directed posets).
Following \cite{AR1994}, an object $F \in \mathcal{G}$
is called $\lambda$-\emph{presentable} if for each $\lambda$-direct system $(Y_i~|~ i \in I)$ in $\mathcal{G}$, the
canonical map
$$\lim\limits_{\longrightarrow} \Hom_{\mathcal{G}}(F, Y_i) \longrightarrow \Hom_{\mathcal{G}}(F, \lim\limits_{\longrightarrow} Y_i)$$
is an isomorphism.
The category $\mathcal{G}$ is called $\lambda$-\emph{accessible}
if there is a set $\mathcal{S}$ of $\lambda$-presentable objects
such that every object of $\mathcal{G}$ is a $\lambda$-direct limit
of objects from $\mathcal{S}$.
An additive category is called \emph{accessible}
if it is $\lambda$-accessible for some regular cardinal $\lambda$.	
Moreover, by a \emph{locally presentable} category we mean a cocomplete accessible category.

\begin{lem}\label{T-1}
	Let $\mathcal{A}$ be a locally presentable exact category
	with cokernels and enough projectives and injectives,
	and $(\mathcal{X}, \mathcal{Y})$ a cotorsion pair in $\mathcal{A}$.
	If $(\mathcal{X}, \mathcal{Y})$ is generated by a set of objects in $\mathcal{A}$, then it is complete.
\end{lem}

\begin{proof}
	Any accessible category is a WIC category by \cite[Observation 2.4 and Remark 1.21]{AR1994}.
	In addition, according to \cite[Proposition A.6]{SS2011} and our assumptions,
	arbitrary transfinite compositions of inflations in $\mathcal{A}$ exist and are themselves inflations.
	
	In view of \cite[Proposition 2.11]{Sto201402},
	it suffices to construct a set $\mathcal{S}$ of objects in $\mathcal{A}$
	such that
	
	(1) $\mathcal{S}$ generates the cotorsion pair $(\mathcal{X}, \mathcal{Y})$, and
	
	(2) any object from $\mathcal{A}$ is a quotient of a coproduct of objects from $\mathcal{S}$
	(in other words, $\mathcal{S}$ is generating).
	
	By the assumption on $(\mathcal{X}, \mathcal{Y})$,
	one can take a set $\mathcal{S}_0$ of objects in $\mathcal{A}$ such that $\mathcal{S}_0^{\perp} = \mathcal{Y}$.
	Suppose that $\mathcal{A}$ is a locally $\lambda$-presentable for some regular cardinal $\lambda$.
	Moreover, we can assume the cardinal $\lambda$ is good enough
	such that any object from $\mathcal{S}_0$ is $\lambda$-presentable
	(see the Remark following \cite[Proposition 1.16]{AR1994}).
	Meanwhile, there is a set $\mathcal{T}$ of $\lambda$-presentable objects
	such that every object in $\mathcal{A}$ is a $\lambda$-direct limit of objects from $\mathcal{T}$.
	Indeed, every object in $\mathcal{A}$ is a quotient of a coproduct of objects from $\mathcal{T}$
	since $\mathcal{A}$ has cokernels and is cocomplete.
	Note that the class $\mathcal{X}$ is generating
	since $\mathcal{A}$ has enough projectives.
	Consequently, any $\lambda$-presentable object $C$ from $\mathcal{T}$
	is a quotient of an object $X_C$ from $\mathcal{X}$.
	Now, add these objects $X_C$ to the set $\mathcal{S}_0$,
	and then we get a new set $\mathcal{S}$
	which generates the cotorsion pair $(\mathcal{X}, \mathcal{Y})$
	and is generating for $\mathcal{A}$.
\end{proof}

\begin{rem}\label{T-2} {\rm
		G\"{o}bel and  Trlifaj \cite[Theorem 3.2.1]{GT2006} proved
		the above result for $\mathcal{A} = R$-Mod,
		the category of modules over an arbitrary ring $R$.
		They also remarked in \cite[Remark 3.2.2(b)]{GT2006}
		that the proof of \cite[Theorem 3.2.1]{GT2006}
		can be extended to any Grothendieck category.}
\end{rem}

\begin{lem}\label{T-3} Let $\mathcal{A}$ be a WIC category with enough projectives and enough injectives  and $(\mathcal{X}, \mathcal{Y})$ a cotorsion pair generated by a set of objects in $\mathcal{A}$.
	Then, for any integer $n\geq 0$, $(^{\perp}\mathcal{Y}_n, \mathcal{Y}_n)$ is a cotorsion pair generated by a set of objects in $\mathcal{A}$.
\end{lem}

\begin{proof}
	Similar to \cite[Lemma 1.13]{C-W2024}.
\end{proof}

Now we can give a sufficient condition for
a cotorsion pair $(\mathcal{X}, \mathcal{Y})$ to be right extendable.

\begin{prop} \label{T-4}
	Let $\mathcal{A}$ be a locally presentable exact category with cokernels
	and enough projectives and injectives,
	and $(\mathcal{X}, \mathcal{Y})$ a cotorsion pair in $\mathcal{A}$.
	If $(\mathcal{X}, \mathcal{Y})$ is generated
	by a set of objects in $\mathcal{A}$,
	then we have a complete cotorsion pair
	$(^{\perp}\mathcal{Y}_n, \mathcal{Y}_n)$ for any integer $n\geq 0$,
	that is, $(\mathcal{X}, \mathcal{Y})$ is right extendable.
\end{prop}

\begin{proof}
	By Lemma \ref{T-3},
	$(^{\perp}\mathcal{Y}_n, \mathcal{Y}_n)$ is a cotorsion pair
	generated by a set in $\mathcal{A}$.
	In addition, it is complete by Lemma \ref{T-1}.
\end{proof}

\begin{rem}\label{T}
	{\rm Let $R$ be a ring and $(\mathcal{X}, \mathcal{Y})$
		a cotorsion pair generated by a set of objects in the category $R$-Mod
		of $R$-modules,
		by \cite[Theorem 2.2]{IEA2012} and \cite[Lemma 1.13]{C-W2024},
		$(\mathcal{X}_n, \mathcal{X}_n^{\perp})$ and $(^{\perp}\mathcal{Y}_n, \mathcal{Y}_n)$
		are cotorsion pairs generated by sets,
		and hence they are complete by \cite[Theorem 3.2.1]{GT2006}.
		Therefore, $(\mathcal{X}, \mathcal{Y})$ is both left and right extendable.
		Moreover, if $(\mathcal{X}, \mathcal{Y})$ is hereditary,
		then $(\mathcal{X}_n, \mathcal{X}_n^{\perp})$ and $(^{\perp}\mathcal{Y}_n, \mathcal{Y}_n)$
		are also hereditary (see Remark \ref{T-2-1}).}
\end{rem}

The following corollary is an immediate consequence of
Proposition \ref{T-4} and Theorem \ref{modle structure 1}.

\begin{cor} \label{T-5}
	Let $\mathcal{A}$ be a locally presentable exact category with cokernels and enough projectives and injectives, and $\mathcal{M}=(\mathcal{C}, \mathcal{W}, \mathcal{F})$ a hereditary Hovey triple in $\mathcal{A}$.
	If $(\mathcal{C}, \mathcal{W} \cap \mathcal{F})$ is generated by a set
	then, for any nonnegative integer $n$, we have the following hereditary Hovey triple
	$$\mathcal{M}_{n}=(^{\perp}(\mathcal{W}\cap\mathcal{F})_{n}, \mathcal{W}, \mathcal{F}_{n})$$
	and the following triangle equivalence
	$$\underline{\mathcal{C}\cap \mathcal{F}}
	\simeq \Ho(\mathcal{M})
	= \Ho(\mathcal{M}_{n})
	\simeq \underline{^{\perp}(\mathcal{W}\cap \mathcal{F})_{n} \cap \mathcal{F}_{n}}\,.
	$$
\end{cor}

By Remark \ref{T} and Theorems \ref{modle structure} and \ref{modle structure 1},
we obtain the following result, which looks symmetric.

\begin{cor} \label{T-6} Let $R$ be a ring and $\mathcal{M}=(\mathcal{C}, \mathcal{W}, \mathcal{F})$ a hereditary Hovey triple in $R\text{-}\Mod$.
	If both $(\mathcal{C}\cap \mathcal{W}, \mathcal{F})$
	and $(\mathcal{C}, \mathcal{W}\cap\mathcal{F})$ are generated by sets
	then, for any nonnegative integer $n$,
	we have the following hereditary Hovey triples
	$$\mathcal{M}_{p,n} = (\mathcal{C}_{n}, \mathcal{W}, (\mathcal{C}\cap \mathcal{W})_{n}^{\perp})
	\quad\text{and}\quad
	\mathcal{M}_{i,n} = (^{\perp}(\mathcal{W}\cap\mathcal{F}) _{n}, \mathcal{W}, \mathcal{F}_{n})$$
	and the following triangle equivalences
	\begin{center}$\begin{aligned}
			\underline{\mathcal{C}_{n} \cap (\mathcal{C}\cap \mathcal{W})_{n} ^{\perp}} \simeq \Ho(\mathcal{M}_{p,n}) &= \Ho(\mathcal{M})\simeq \underline{\mathcal{C}\cap \mathcal{F}}\\
			&=\Ho(\mathcal{M}_{i,n})\simeq \underline{^{\perp}(\mathcal{W}\cap \mathcal{F})_{n} \cap \mathcal{F}_{n}}
		\end{aligned}$\end{center}
\end{cor}

\section{\bf Some applications}
\subsection {$Q$-shaped derived category}
In this section,
we assume that $\mathbf{k}$ is a non-trivial commutative Gorenstein ring,
$R$ is an associative $\mathbf{k}$-algebra with finite projective dimension as a $\mathbf{k}$-module.

\begin{setup} \label{setup}
	\emph{\cite[Setup 2.5]{HJ2022} Let $Q$ be a small $\mathbf{k}$-pre-additive category satisfying the following conditions.}
	
	\emph{(1)} Hom-finiteness\emph{: each hom $\mathbf{k}$-module $Q(p, q)$ is finitely generated and projective.}
	
	\emph{(2)} Local Boundedness\emph{: for each $q \in Q$, there are only finitely many objects in $Q$ mapping non-trivially into or out of $q$.}
	
	\emph{(3)} Existence of a Serre Functor \emph{(}relative to $\mathbf{k}$\emph{): there exists a $\mathbf{k}$-linear auto-equivalence
		$\mathbb{S} : Q \to Q$ and a natural isomorphism $Q(p, q) \cong \Hom_{\mathbf{k}}(Q(q, \mathbb{S}(p)), \mathbf{k})$.}
	
	\emph{(4)} Retraction Property\emph{: for each $q \in Q$,
		the unit map $\mathbf{k} \to Q(q, q)$ given by $x \mapsto x \cdot \id_q$ has a $\mathbf{k}$-module retraction;
		whence there is a $\mathbf{k}$-module decomposition $Q(q, q) = (\mathbf{k} \cdot \id_q) \oplus \tau _q$.}
\end{setup}
\noindent Denote by $_{Q, R}$Mod the category of $\mathbf{k}$-linear functors from $Q$ to the left $R$-module
category $R$-Mod.
If $R = \mathbf{k}$, then we set $_Q$Mod instead of $_{Q, \mathbf{k}}$Mod.
In view of \cite[Proposition 3.12]{HJ2022}, $_{Q, R}$Mod is a locally finitely presentable Grothendieck category with enough projectives.
The $Q$-shaped derived category $\mathbf{D}_Q(R)$ of $R$,
introduced by Holm and J{\o}rgensen \cite{HJ2022},
is defined as the homotopy category of the model category $_{Q, R}$Mod,
which is a generalization of derived category of algebra.
See \cite{HJ2024} for a quick introduction.

\begin{lem} \label{IF} \emph{\cite[Theorem 6.1, 6.5]{HJ2022}}
	Assume that $Q$ satisfies the conditions in \emph{Setup \ref{setup}}, $\mathbf{k}$ is commutative and Gorenstein,
	and $R$ has finite projective dimension as a $\mathbf{k}$-module.
	There exists a hereditary Hovey triple
	$$\mathcal{M} = (_{Q, R}\Mod, \mathcal{E}, \mathcal{E}^{\perp})$$
	in $_{Q, R}\Mod$, where $\mathcal{E} = \{X {}\in _{Q, R}\Mod ~|~ X^{\natural} \text{ has finite projective dimension in } _Q\Mod\}$ and $(-)^{\natural}$ denotes the forgetful functor from $_{Q, R}\Mod$ to $_Q\Mod$.
	Moreover, there are equivalences of categories,
	$$\underline{\mathcal{E}^{\perp}} \simeq \Ho(\mathcal{M}) \simeq \mathbf{D}_Q(R).$$
\end{lem}

We note that $(_{Q, R}\Mod, \mathcal{E} \cap \mathcal{E}^{\perp}) = (_{Q, R}\Mod, \mathrm{Inj}(Q, R))$ is right extendable by \cite[Proposition 5.3(1)]{HJ2022},
where $\mathrm{Inj}(Q, R)$ denotes the class of injective objects in $_{Q, R}\Mod$.
Hence one can obtain a new description of $Q$-shaped derived category by Theorem \ref{modle structure 1}.

\begin{thm} \label{DQ}
	Keep the conditions and notations in Lemma \ref{IF},
	then for all $n \geq 0$ we have the following hereditary Hovey triples
	$$\mathcal{M}_n = (^{\perp}(\mathrm{Inj}(Q, R)_n), \mathcal{E}, (\mathcal{E}^{\perp})_n),$$
	and the following triangle equivalences
	$$\underline{^{\perp}(\mathrm{Inj}(Q, R)_n) \cap (\mathcal{E}^{\perp})_n} \simeq \Ho(\mathcal{M}_n) \simeq \mathbf{D}_Q(R).$$
\end{thm}

\begin{exa} \label{D}
	    \emph{Given a ring $R$, let $R$-\Mod \ and \Ch$(R)$ be the categories of
		left $R$-modules and complexes of left $R$-modules, respectively.
		The classes of all injective complexes,
		DG-injective complexes,
		and exact complexes in \Ch$(R)$
		are denoted by $\widetilde{\mathcal{I}}$,
		\dg$(\mathcal{I})$,
		and $\mathbf{Acy}$, respectively.
		A chain complex can be viewed as a representation of the linear quiver
		$$\Gamma = \cdots \to 2 \to 1 \to 0 \to -1 \to -2 \to \cdots$$
		with the relations that consecutive arrows compose to $0$.
		By taking $Q$ to be the path category of $\Gamma$ with these relations and $\mathbf{k} = \mathbb{Z}$,
		it follows that $_{Q, R}\Mod$ is equivalent to the category \Ch$(R)$ (see, e.g., \cite{HJ2024}).
		Moreover, Lemma \ref{IF} includes a classic hereditary Hovey triple $\mathcal{M}_{\mathrm{dg}}
		= (\Ch(R), \mathbf{Acy}, \mathrm{dg}(\mathcal{I}))$ in \Ch$(R)$ as an example,
		whose homotopy category is the ordinary derived category $\mathbf{D}(R)$ of $R$
		(see, e.g., \cite[Example 3.2]{Hov2002}).
		Therefore, by Theorem \ref{DQ}, for all integers $n\geq 0$,
		there are hereditary Hovey triples
		$$\mathcal{M}_{\mathrm{dg}, n} =
		(^{\perp}(\widetilde{\mathcal{I}}_n), \mathbf{Acy}, (\mathrm{dg}(\mathcal{I}))_n)$$
		with triangle equivalences
		$$\underline{^{\perp}(\widetilde{\mathcal{I}}_n) \cap (\mathrm{dg}(\mathcal{I}))_n}
		\simeq \Ho(\mathcal{M}_{\mathrm{dg}, n})
		\simeq \mathbf{D}(R).$$
		We remind the reader that $(\mathrm{dg}(\mathcal{I}))_n$
		is different from $\mathrm{dg}(\mathcal{I}_n)$,
		where $\mathcal{I}_n$ stands for the class of left $R$-modules of injective dimension less than or equal to $n$.
		In general, for a cotorsion pair $({}^{\perp}\mathcal{Y}, \mathcal{Y})$ of left $R$-modules,
		one might be interested in the class $\mathrm{dg}(\mathcal{Y})$.
		See \cite{Gil04} for more details.}
\end{exa}

\subsection {Krause's recollement}

Let $R$ be a ring.
It is well known that both $R$-\Mod \ and \Ch$(R)$
are locally presentable abelian categories
with enough projectives and injectives.
The class of all complexes of injective $R$-modules is denoted by dw($\mathcal{I}$).
In addition,
we write ex($\mathcal{I}$) = dw$(\mathcal{I}) \cap \mathbf{Acy}$.
For a subcategory $\mathcal{H}$ of \Ch$(R)$,
the homotopy category of $\mathcal{H}$ will be denoted by
K($\mathcal{H}$).

Over a Noetherian ring $R$, Krause \cite{Kra2005} established the following recollement
\[
\xymatrixrowsep{5pc}
\xymatrixcolsep{6pc}
\xymatrix
{\K(\mathrm{ex}(\mathcal{I}))  \ar[r]^{} &\K(\mathrm{dw}(\mathcal{I})) \ar@<-3ex>[l]_{} \ar@<3ex>[l]^{} \ar[r]^{}  &\mathbf{D}(R)\ar@<-3ex>[l]_{} \ar@<3ex>[l]^{},
} \]
which is called Krause's recollement in the literature.
By the work of Becker \cite{Bec2014} or Gillespie \cite{Gil2016},
we know that such a recollement holds for any ring $R$.

The following fact is well-known (see, e.g., \cite[Example 4.8]{Gil2016}).

\begin{fact}\label{KR-0}
	Let $R$ be a ring. There are hereditary Hovey triples
	\begin{center}
		$\begin{aligned}
			\mathcal{M}_{\mathrm{dw}} =
			(\Ch(R), {}^{\perp}(\mathrm{dw}(\mathcal{I})), \mathrm{dw}(\mathcal{I}))
			\quad and \quad
			\mathcal{M}_{\mathrm{ex}} =
			(\Ch(R), {}^{\perp}(\mathrm{ex}(\mathcal{I})), \mathrm{ex}(\mathcal{I}))
		\end{aligned}$ \end{center}
	with the following triangle equivalences
	$$\Ho(\mathcal{M}_{\mathrm{dw}})\simeq \K(\mathrm{dw}(\mathcal{I}))
	\quad and \quad
	\Ho(\mathcal{M}_{\mathrm{ex}})\simeq \K(\mathrm{ex}(\mathcal{I})).$$
\end{fact}

Note that the cotorsion pair $(\Ch(R), \widetilde{\mathcal{I}})$
is generated by a set,
from Fact \ref{KR-0} and Corollary \ref{T-5},
we obtain the following lemma.

\begin{lem}\label{KR-i0}
	Let $R$ be a ring then, for any integer $n\geq 0$, there are hereditary Hovey triples
	$$\mathcal{M}_{\mathrm{dw},n} =
	(^{\perp}(\widetilde{I}_n), {}^{\perp}(\mathrm{dw}(\mathcal{I})), (\mathrm{dw}(\mathcal{I}))_n)
	\quad and \quad
	\mathcal{M}_{\mathrm{ex},n} =
	(^{\perp}(\widetilde{I}_n), {}^{\perp}(\mathrm{ex}(\mathcal{I})), (\mathrm{ex}(\mathcal{I}))_n)$$
	with the following triangle equivalences
	$$\Ho(\mathcal{M}_{\mathrm{dw},n})\simeq \K(\mathrm{dw}(\mathcal{I}))
	\quad and \quad
	\Ho(\mathcal{M}_{\mathrm{ex},n})\simeq \K(\mathrm{ex}(\mathcal{I})).$$
\end{lem}

Combining Example \ref{D}, Fact \ref{KR-0} and Lemma \ref{KR-i0},
we can interpret the aforementioned Krause's recollement
in terms of ``$n$-dimensional'' homotopy categories ($n \geq 0$).

\begin{prop}\label{KR-i}
	Let $R$ be a ring and $n\geq 0$ an integer. Keep the notations in Example \ref{D} and Lemma \ref{KR-i0}.
	Then we have a recollement of homotopy categories of hereditary Hovey triples:
	\[
	\xymatrixrowsep{5pc}
	\xymatrixcolsep{6pc}
	\xymatrix
	{\Ho(\mathcal{M}_{\mathrm{ex},n}) \ar[r] & \Ho(\mathcal{M}_{\mathrm{dw},n}) \ar@<-3ex>[l] \ar@<3ex>[l] \ar[r] & \Ho(\mathcal{M}_{\mathrm{dg}, n}). \ar@<-3ex>[l] \ar@<3ex>[l]
	} \]
\end{prop}

More generally,
combining Theorem \ref{modle structure 1} and \cite[Theorem 3.4]{Gil16},
one can obtain the following

\begin{prop}\label{gKR-i}
	Let $\mathcal{A}$ be a WIC exact category with enough projectives and injectives.
	Suppose that we have three hereditary Hovey triples
	$$\mathcal{M}_1=(\mathcal{A}, \mathcal{W}_1, \mathcal{F}_1),\quad
	\mathcal{M}_2=(\mathcal{A}, \mathcal{W}_2, \mathcal{F}_2), \quad
	\mathcal{M}_3=(\mathcal{A}, \mathcal{W}_3, \mathcal{F}_3)  $$
	such that
	$$\mathcal{W}_3\cap \mathcal{F}_1 = \mathcal{F}_2 \ \ \text{and} \ \
	\mathcal{F}_3\subseteq \mathcal{F}_1 \quad(\text{or equivalently},
	\mathcal{W}_2\cap \mathcal{W}_3=\mathcal{W}_1 \ \ \text{and} \ \
	\mathcal{F}_2\subseteq \mathcal{W}_3).$$
	If the cotorsion pair $(\mathcal{A}, \mathcal{A}^{\perp})$
	is right extendable then, for any integer $n\geq 0$,
	we have hereditary Hovey triples
	$$\mathcal{M}_{i,n}
	=(^{\perp}(\mathcal{W}_i \cap \mathcal{F}_i)_n, \mathcal{W}_i, (\mathcal{F}_i)_n),
	\quad (i = 1, 2, 3)$$
	whose homotopy categories fit in a recollement:
	\[
	\xymatrixrowsep{5pc}
	\xymatrixcolsep{6pc}
	\xymatrix
	{\mathrm{Ho}(\mathcal{M}_{2,n}) \ar[r] & \mathrm{Ho}(\mathcal{M}_{1,n}) \ar@<-3ex>[l] \ar@<3ex>[l] \ar[r] &\mathrm{Ho}(\mathcal{M}_{3,n}).\ar@<-3ex>[l] \ar@<3ex>[l]
	} \]
\end{prop}

To avoid being repetitious, we left the dual statement of Proposition \ref{gKR-i} to the reader.

\subsection {Model structures in the representation category of a quiver}	

A quiver $Q$ is actually a directed graph with a set $Q_0$ of vertexes
and a set $Q_1$ of arrows.
For an arrow $a \in Q_1$,
we always write $s(a)$ for its source and $t(a)$ for its target.
For a vertex $i \in Q_0$,
we set $Q^{i\to *} = \{a\in Q_1~|~s(a)=i\}$
and $Q_1^{*\to i} = \{a\in Q_1~|~t(a)=i\}$.
According to \cite[Proposition 3.6]{E-T2004},
we shall say a quiver is \emph{left rooted}
if it has no infinite sequence of arrows of the form
$\cdots \to \bullet \to \bullet \to \bullet$.
A \emph{right rooted} quiver is defined dually.

Let $\mathcal{A}$ be an abelian category and $Q$ be a quiver.
We denote by Rep$(Q,\mathcal{A})$ the category of
all representations of $Q$, which is an abelian category.
Here a \emph{representation} $X$ of $Q$ is a covariant functor $X: Q\to \mathcal{A}$. And a morphism $X\to Y$ of representations is a natural transformation, that is, a family $\{f(i): X(i)\to Y(i)\}_{i\in Q_0}$ such that the following diagram is commutative for each arrow $a:i\to j$ in $Q_1$:
$$\xymatrix{
	X(i) \ar[d]_{f(i)} \ar[r]^{X(a)}
	& X(j) \ar[d]^{f(j)}  \\
	Y(i)  \ar[r]_{Y(a)}
	& Y(j)             }
$$

Recall that an abelian category $\mathcal{A}$ satisfies the axiom AB3 provided that it has small coproducts. It satisfies AB4 provided that it satisfies AB3 and any coproduct of monomorphisms is a monomorphism.
Dually, one has the axioms AB3$^{*}$ and AB4$^{*}$.

\begin{lem} \label{Rep-1} \cite[Corollary 3.10]{HJ2019}
	Let $\mathcal{A}$ be any abelian category and let $Q$ be any quiver.
	
	\emph{(1)}	Assume that $\mathcal{A}$ satisfies {\rm AB3}. If $\mathcal{A}$ has enough projectives, then so does \Rep$(Q,\mathcal{A})$.
	
	\emph{(2)}	Assume that $\mathcal{A}$ satisfies {\rm AB3$^{*}$}. If $\mathcal{A}$ has enough injectives, then so does \Rep$(Q,\mathcal{A})$.
	
\end{lem}

For an abelian category $\mathcal{A}$ and a representation
$X\in$ Rep$(Q,\mathcal{A})$,
using the universal property of coproducts (resp., products),
there is a unique morphism
$$\varphi_i^{X}: \underset{a\in Q_1^{^{*\to i}}}\coprod X(s(a))\to X(i)
\quad (\mathrm{resp.}, \psi_i^{X}: X(i) \to\underset{a\in Q_1^{^{i\to *}}}\prod X(t(a)))$$
for each $i\in Q_0$.
We denote by C$_i(X)$ and K$_i(X)$ the cokernel of $\varphi_i^{X}$
and kernel of $\psi_i^{X}$, respectively.
For any subcategory $\mathcal{X}$ of $\mathcal{A}$, we put
\begin{center}$\begin{aligned}&\bullet~ \text{Rep}(Q, \mathcal{X})=\{X\in \text{Rep}(Q, \mathcal{A})~|~X(i)\in \mathcal{X}, \forall i\in Q_0\},\\
		&\bullet~ \Phi(\mathcal{X})=\{X\in \text{Rep}(Q, \mathcal{A})~|~\varphi_i^{X}\ \text{is a monomorphism,}\ \mathrm{C}_i(X) \in \mathcal{X}, \forall i\in Q_0 \}, \text{and}\\
		&\bullet~ \Psi(\mathcal{X})=\{X\in \text{Rep}(Q, \mathcal{A})~|~\psi_i^{X}\ \text{is a epimorphism,}\ \mathrm{K}_i(X)~\in \mathcal{X}, \forall i\in Q_0 \}.
	\end{aligned}$ \end{center}

\begin{lem} \label{Rep-2} \cite[Corollary 3.5]{D-X2024}
	Let $\mathcal{A}$ be an abelian category, and let $(\mathcal{X}, \mathcal{Y})$ be a complete cotorsion pair in $\mathcal{A}$.
	
	\emph{(1)}	Suppose that $Q$ is a left rooted quiver and $\mathcal{A}$ satisfies {\rm AB4}. Then $(\Phi(\mathcal{X})$, $\Rep(Q, \mathcal{Y}))$ is a complete cotorsion pair in \Rep$(Q,\mathcal{A})$.
	
	\emph{(2)}	Suppose that $Q$ is a right rooted quiver and $\mathcal{A}$ satisfies {\rm AB4$^{*}$}. Then $(\Rep(Q, \mathcal{X}), \Psi(\mathcal{Y}))$ is a complete cotorsion pair in \Rep$(Q,\mathcal{A})$.
	
\end{lem}

\begin{lem} \label{Rep-3}
	Let $\mathcal{A}$ be an abelian category satisfying {\rm AB4},
	$Q$ a left rooted quiver and
	$(\mathcal{X}, \mathcal{Y})$ a complete cotorsion pair in $\mathcal{A}$.
	If $(\mathcal{X}, \mathcal{Y})$ is left extendable,
	then so is $(\Phi(\mathcal{X}), \Phi(\mathcal{X})^{\perp})$.
\end{lem}
\begin{proof}
	We will show that $\Phi(\mathcal{X}_n) = \Phi(\mathcal{X})_n$, which implies that $(\Phi(\mathcal{X}), \Phi(\mathcal{X})^{\perp})$ is left extendable by Definition \ref{D-1} and Lemma \ref{Rep-2}.
	The case when $n = 0$ is clear.
	Now assume $n \geq 1$.
	
	For any $X \in \Phi(\mathcal{X})_n$,
	we can take an exact sequence
	$$0 \to B_n \to B_{n-1} \to \cdots \to B_0 \to X \to 0$$
	where $B_l \in \Phi(\mathcal{X})$, $0 \leq l \leq n$.
	There are two cases:
	
	Case 1: for a source $i$,
	we have $\underset{a\in Q_1^{^{*\to i}}}\coprod Y(s(a)) = 0$
	for any $Y \in {\rm Rep}(Q,\mathcal{A})$.
	Then we have the following commutative diagram
	\begin{displaymath}
		\xymatrix{
			&0 \ar[d] & &0 \ar[d] &0 \ar[d] & \\
			0 \ar[r]& \underset{a\in Q_1^{^{*\to i}}}\coprod B_n(s(a)) \ar[d] \ar[r] & \cdots \ar[r] & \underset{a\in Q_1^{^{*\to i}}}\coprod B_0(s(a)) \ar[d] \ar[r] & \underset{a\in Q_1^{^{*\to i}}}\coprod X(s(a)) \ar[d]^{\varphi_i^{X}} \ar[r] & 0 \\
			0 \ar[r]& B_n(i)  \ar[d] \ar[r] & \cdots \ar[r] &B_0(i)  \ar[d] \ar[r] &X(i)  \ar[d] \ar[r] &0 \\
			0 \ar[r]& B_n(i)  \ar[d] \ar[r] & \cdots \ar[r] &B_0(i)  \ar[d] \ar[r] &X(i)  \ar[d] \ar[r] &0 \\
			&0  & &0  &0  & \\
		}
	\end{displaymath}
Since $\underset{a\in Q_1^{^{*\to i}}}\coprod Y(s(a)) = 0$
	for any $Y \in {\rm Rep}(Q,\mathcal{A})$
	and $B_l \in \Phi(\mathcal{X}), 0 \leq l \leq n$,
	then $\varphi_i^{X}$ is a monomorphism and $\mathrm{C}_i(X) = X(i) \in \mathcal{X}_n$.
	
	Case 2: for a vertex $i$ which is not a source,
	we have the following commutative diagram
	\begin{displaymath}
		\xymatrix{
			&0 \ar[d] & &0 \ar[d] & & \\
			0 \ar[r]& \underset{a\in Q_1^{^{*\to i}}}\coprod B_n(s(a)) \ar[d] \ar[r] & \cdots \ar[r] & \underset{a\in Q_1^{^{*\to i}}}\coprod B_0(s(a)) \ar[d] \ar[r] & \underset{a\in Q_1^{^{*\to i}}}\coprod X(s(a)) \ar[d]^{\varphi_i^{X}} \ar[r] & 0 \\
			0 \ar[r]& B_n(i)  \ar[d] \ar[r] & \cdots \ar[r] &B_0(i)  \ar[d] \ar[r] &X(i)  \ar[d] \ar[r] &0 \\
			0 \ar[r]& C_n(i)  \ar[d] \ar[r] & \cdots \ar[r] &C_0(i)  \ar[d] \ar[r] &C(i)  \ar[d] \ar[r] &0 \\
			&0  & &0  &0  & \\
		}
	\end{displaymath}
from which we can see that $\varphi_i^{X}$ is a monomorphism
	by the Five lemma.
	Since $B_l \in \Phi(\mathcal{X})$, $0 \leq l \leq n$,
	then $C_l(i) \in \mathcal{X}$ and hence
	$\mathrm{C}_i(X) = C(i) \in \mathcal{X}_n$.
	Thus we have $\Phi(\mathcal{X})_n \subseteq \Phi(\mathcal{X}_n)$.
	
	On the other hand, we will show that
	$\Phi(\mathcal{X}_n) \subseteq \Phi(\mathcal{X})_n$ by induction on $n$.
	Assume $n \geq 1$ and $\Phi(\mathcal{X}_{n-1}) \subseteq \Phi(\mathcal{X})_{n-1}$.
	If $X \in \Phi(\mathcal{X}_n)$,
	by Lemma \ref{Rep-2}(1),
	we can construct an exact sequence
	$$0 \to B_n \to B_{n-1} \to \cdots \to B_0 \to X \to 0$$
	where $B_l \in \Phi(\mathcal{X}), 0 \leq l \leq n-1$.
	Then we have a short exact sequence $0 \to K_0 \to B_0 \to X \to 0$.
	There are two cases:
	
	Case 1: for a source $i$, we have the following commutative diagram
	\begin{displaymath}
		\xymatrix{
			& 0 \ar[d]&0 \ar[d] &0 \ar[d] & \\
			0 \ar[r]& \underset{a\in Q_1^{^{*\to i}}}\coprod K_0(s(a)) \ar[d]^{\varphi_i^{K_0}} \ar[r] & \underset{a\in Q_1^{^{*\to i}}}\coprod B_0(s(a)) \ar[d] \ar[r] & \underset{a\in Q_1^{^{*\to i}}}\coprod X(s(a)) \ar[d] \ar[r] & 0 \\
			0 \ar[r]& K_0(i)  \ar[d] \ar[r] &B_0(i)  \ar[d] \ar[r] &X(i)  \ar[d] \ar[r] &0 \\
			0 \ar[r]& K_0(i)  \ar[d] \ar[r] &B_0(i)  \ar[d] \ar[r] &X(i)  \ar[d] \ar[r] &0 \\
			&0  &0  &0  & \\
		}
	\end{displaymath}
	Since $\underset{a\in Q_1^{^{*\to i}}}\coprod Y(s(a)) = 0$
	for any $Y \in {\rm Rep}(Q,\mathcal{A})$ and
	$B_0 \in \Phi(\mathcal{X})$, then $\varphi_i^{K_0}$ is a monomorphism
	and $\mathrm{C}_i(K_0) = K_0(i) \in \mathcal{X}_{n-1}$ by Lemmma \ref{06-2}(1).
	
	Case 2: for a vertex $i$ which is not a source,
	we have the following commutative diagram
	\begin{displaymath}
		\xymatrix{
			& &0 \ar[d] &0 \ar[d] & \\
			0 \ar[r]& \underset{a\in Q_1^{^{*\to i}}}\coprod K_0(s(a)) \ar[d]^{\varphi_i^{K_0}} \ar[r] & \underset{a\in Q_1^{^{*\to i}}}\coprod B_0(s(a)) \ar[d] \ar[r] & \underset{a\in Q_1^{^{*\to i}}}\coprod X(s(a)) \ar[d] \ar[r] & 0 \\
			0 \ar[r]& K_0(i)  \ar[d] \ar[r] &B_0(i)  \ar[d] \ar[r] &X(i)  \ar[d] \ar[r] &0 \\
			0 \ar[r]& D_0(i)  \ar[d] \ar[r] &C_0(i)  \ar[d] \ar[r] &C(i)  \ar[d] \ar[r] &0 \\
			&0  &0  &0  & \\
		}
	\end{displaymath}
	from which we can see $\varphi_i^{K_0}$ is a monomorphism
	by the Five lemma
	and $\mathrm{C}_i(K_0) = D_0(i) \in \mathcal{X}_{n-1}$
	by Lemmma \ref{06-2}(1).
	Then $K_0 \in \Phi(\mathcal{X}_{n-1}) \subseteq \Phi(\mathcal{X})_{n-1}$,
	and hence $X \in \Phi(\mathcal{X})_n$.
	This completes the proof.
\end{proof}

Let $Q$ be a left rooted quiver and let $\mathcal{A}$ be an abelian category
satisfying {\rm AB4} and {\rm AB4${^*}$} with enough projectives and injectives,
then Rep$(Q,\mathcal{A})$ is an abelian category with enough projectives and injectives
by Lemma \ref{Rep-1}.
If there is a hereditary Hovey triple
$\mathcal{M} = (\mathcal{C}, \mathcal{W}, \mathcal{F})$ in $\mathcal{A}$
such that the cotorsion pair $(\mathcal{C} \cap \mathcal{W}, \mathcal{F})$
is left extendable,
then there is a chain of hereditary Hovey triples
$$\mathcal{M}_{\mathrm{p},n} = (\mathcal{C}_{n}, \mathcal{W}, (\mathcal{C}\cap \mathcal{W})_{n} ^{\perp})
\quad (n \geq 0)$$
by Theorem \ref{modle structure}.
Moreover,
$\mathcal{M}$ and $\mathcal{M}_{\mathrm{p},n}$ in $\mathcal{A}$
induce respectively the hereditary Hovey triples
$$\Rep\text{-}(\mathcal{M}_{p}) = (\Phi(\mathcal{C}), \Rep(Q,\mathcal{W}), \Rep(Q, F))$$
and
$$\Rep\text{-}(\mathcal{M}_{p,n})=(\Phi(\mathcal{C}_n), \Rep(Q,\mathcal{W}), \Rep(Q, (\mathcal{C}\cap \mathcal{W})_{n} ^{\perp}))$$
in Rep$(Q,\mathcal{A})$ (see \cite[Theorem B]{D-O2021}).
Furthermore, the cotorsion pair
$$(\Phi(\mathcal{C}) \cap \Rep(Q,\mathcal{W}), \Rep(Q, F))
= (\Phi(\mathcal{C} \cap \mathcal{W}), \Rep(Q, F))$$
is left extendable by Lemmas \ref{Rep-2} and \ref{Rep-3},
then there is a chain of hereditary Hovey triples
$$(\Phi(\mathcal{C})_n, \Rep(Q,\mathcal{W}), (\Phi(\mathcal{C}) \cap \Rep(Q, \mathcal{W}))_{n} ^{\perp}))
\quad (n \geq 0)$$
in Rep$(Q,\mathcal{A})$ by Theorem \ref{modle structure}.
Now we have the following diagram:
\begin{displaymath}
	\xymatrix{
		\mathcal{M} = (\mathcal{C}, \mathcal{W}, \mathcal{F}) \ar[r]^{\mathrm{[6, Th\ B]}\hspace{26mm}} \ar[dd]^{\mathrm{Th\ 3.6}} & \Rep\text{-}(\mathcal{M}_{p})=(\Phi(\mathcal{C}), \Rep(Q,\mathcal{W}), \Rep(Q, F)) \ar[d]^{\mathrm{Th\ 3.6}} \\
		& (\Phi(\mathcal{C})_n, \Rep(Q,\mathcal{W}), (\Phi(\mathcal{C}) \cap \Rep(Q, \mathcal{W}))_{n} ^{\perp}) \ar@{==}[d]^{?} \\
		\mathcal{M}_{p,n} = (\mathcal{C}_{n}, \mathcal{W}, (\mathcal{C}\cap \mathcal{W})_{n} ^{\perp}) \ar[r]^{\mathrm{[6, Th\ B]}\hspace{24mm}} & \Rep\text{-}(\mathcal{M}_{p,n})=(\Phi(\mathcal{C}_n), \Rep(Q,\mathcal{W}), \Rep(Q, (\mathcal{C}\cap \mathcal{W})_{n} ^{\perp}))
	}
\end{displaymath}
From the proof of Lemma \ref{Rep-3} we notice that
$\Phi(\mathcal{C}_n) = \Phi(\mathcal{C})_n$,
which implies the following equality
$$(\Phi(\mathcal{C}_n), \Rep(Q,\mathcal{W}), \Rep(Q, (\mathcal{C}\cap \mathcal{W})_{n} ^{\perp})) 
 = (\Phi(\mathcal{C})_n, \Rep(Q,\mathcal{W}), (\Phi(\mathcal{C}) \cap \Rep(Q, \mathcal{W}))_{n} ^{\perp}).
$$
Finally, in the above diagram,
the homotopy categories of each column of the hereditary Hovey triples
are triangle equivalent by Theorem \ref{modle structure}.
Therefore, we have the following result.

\begin{thm}\label{Q-p}
	Let $Q$ be a left rooted quiver and
	let $\mathcal{A}$ be an abelian category
	satisfying {\rm AB4} and {\rm AB4${^*}$}
	with enough projectives and injectives.
	If there is a hereditary Hovey triple
	$$\mathcal{M} = (\mathcal{C}, \mathcal{W}, \mathcal{F})$$
	in $\mathcal{A}$ such that the cotorsion pair
	$(\mathcal{C} \cap \mathcal{W}, \mathcal{F})$ is left extendable
	then, for any integer $n\geq 0$,
	there is a hereditary Hovey triple in $\mathcal{A}$:
	$$\mathcal{M}_{\mathrm{p},n} = (\mathcal{C}_{n}, \mathcal{W}, (\mathcal{C}\cap \mathcal{W})_{n} ^{\perp}),$$
	which induces a hereditary Hovey triple in \Rep$(Q,\mathcal{A})$:
	\begin{eqnarray*}
		\Rep\text{-}(\mathcal{M}_{p,n}) & = & (\Phi(\mathcal{C}_n), \Rep(Q,\mathcal{W}), \Rep(Q, (\mathcal{C}\cap \mathcal{W})_{n} ^{\perp})) \\
		& = & (\Phi(\mathcal{C})_n, \Rep(Q,\mathcal{W}), (\Phi(\mathcal{C}) \cap \Rep(Q, \mathcal{W}))_{n} ^{\perp})
	\end{eqnarray*}
	with triangle equivalences
	\begin{eqnarray*}
		&   & \Ho(\Rep\text{-}\mathcal{M}_{\mathrm{p},0}) \hspace{3.4mm}
		\simeq \underline{\Phi(\mathcal{C})\cap \Rep(Q,\mathcal{F})} \\
		& = & \Ho(\Rep\text{-}(\mathcal{M}_{\mathrm{p},n}))
		\simeq \underline{\Phi(\mathcal{C}_n)\cap \Rep(Q, (\mathcal{C}\cap \mathcal{W})_{n} ^{\perp})}.
	\end{eqnarray*}
\end{thm}	

Dually, we have

\begin{thm}\label{Q-i}
	Let $Q$ be a right rooted quiver and
	let $\mathcal{A}$ be an abelian category
	satisfying {\rm AB4} and {\rm AB4${^*}$}
	with enough projectives and injectives.
	If there is a hereditary Hovey triple
	$$\mathcal{M} = (\mathcal{C}, \mathcal{W}, \mathcal{F})$$
	in $\mathcal{A}$ such that the cotorsion pair
	$(\mathcal{C}, \mathcal{W} \cap \mathcal{F})$ is right extendable
	then, for any integer $n\geq 0$,
	there is a hereditary Hovey triple in $\mathcal{A}$:
	$$\mathcal{M}_{\mathrm{i},n} = (^{\perp}(\mathcal{W}\cap\mathcal{F})_{n}, \mathcal{W}, \mathcal{F}_{n}),$$
	which induces a hereditary Hovey triple in \Rep$(Q,\mathcal{A})$:
	\begin{eqnarray*}
		\Rep\text{-}(\mathcal{M}_{i,n}) & = & (\Rep(Q, {}^{\perp}(\mathcal{W}\cap \mathcal{F})_{n}), \Rep(Q,\mathcal{W}), \Psi(\mathcal{F}_n)) \\
		& = & (^{\perp}(\Rep(Q,\mathcal{W}) \cap \Psi(\mathcal{F}))_n, \Rep(Q,\mathcal{W}), \Psi(\mathcal{F})_n)
	\end{eqnarray*}
	with triangle equivalences
	\begin{eqnarray*}
		&   & \Ho(\Rep\text{-}\mathcal{M}_{\mathrm{i},0}) \hspace{3.4mm}
		\simeq \underline{\Rep(Q,\mathcal{C})\cap \Psi(\mathcal{F})} \\
		& = & \Ho(\Rep\text{-}(\mathcal{M}_{\mathrm{i},n}))
		\simeq \underline{\Rep(Q, {}^{\perp}(\mathcal{W}\cap \mathcal{F})_{n})\cap \Psi(\mathcal{F}_n)}.
	\end{eqnarray*}
\end{thm}

\bigskip \centerline {\bf ACKNOWLEDGEMENTS}
This work is supported by the National Natural Science Foundation of China
(Grant No. 12361008), the Jiangsu Provincial Scientific Research Center of Applied Mathematics (No. BK20233002),
and the Foundation for Innovative Fundamental Research Group Project of Gansu Province (Grant no. 23JRRA684).

\end{document}